\documentclass[12pt,reqno]{amsart}
\usepackage{amsmath,amstext,amssymb,amscd,hyperref}
\usepackage{verbatim}
\usepackage{enumerate}
\usepackage{mathrsfs}
\usepackage[dvipsnames,usenames]{xcolor}

\usepackage{amsthm}

\newtheorem{theorem}{Theorem}[section]

\newtheorem{corollary}[theorem]{Corollary}

\newtheorem{assumption}[theorem]{Assumption}

\theoremstyle{definition}
\newtheorem{definition}[theorem]{Definition}

\theoremstyle{remark}
\newtheorem{remark}[theorem]{Remark}

\numberwithin{equation}{section}

\def\a{\alpha}

\def\t{\tau}

\def\G{\Gamma}
\def\O{\Omega}

\newcommand{\reals}{\mathbb{R}}

\newcommand{\norm}[1]{\left\|#1\right\|}

\author[Y. Guo]{Yanqiu Guo}
\address{Department of Mathematics \& Statistics \\Florida International University\\
Miami, Florida 33199, USA} \email{yanqiu.guo@fiu.edu}

\author[M. A. Rammaha]{Mohammad A. Rammaha}
\address{Department of
Mathematics \\ University of Nebraska-Lincoln \\
Lincoln, NE  68588-0130, USA} \email{mrammaha1@unl.edu}

\author[S. Sakuntasathien]{Sawanya Sakuntasathien}
\address{ Department of Mathematics \\ Faculty of Science \\ Silpakorn University  \\
Nakhonpathom, 73000, Thailand} \email{sakuntasathien\_s@silpakorn.edu}

\title[Blow-up of hyperbolic equations of viscoelasticity]
{Blow-up of a hyperbolic equation of viscoelasticity with supercritical nonlinearities}
\date{July 21, 2016}
\keywords{viscoelasticity, memory, integro-differential, damping, source, blow-up}
\subjclass[2010]{35L10, 35L20, 35L70, 35B44}

\begin{document}
\maketitle

\begin{abstract}
We investigate a hyperbolic PDE, modeling wave propagation in viscoelastic media, under the influence of a linear memory term of Boltzmann type, and a nonlinear damping modeling friction, as well as an energy-amplifying \emph{supercritical} nonlinear source:
\begin{align*}
\begin{cases}
u_{tt}- k(0) \Delta u -  \int_0^{\infty} k'(s) \Delta u(t-s) ds + |u_t|^{m-1}u_t=|u|^{p-1}u,   \text{\;\;\;\;in\;\;} \O \times (0,T), \\
u(x,t)=u_0(x,t), \quad \text{ in } \O \times (-\infty,0],
\end{cases}
\end{align*}
where $\O$ is a bounded domain in $\mathbb R^3$ with a Dirichl\'et boundary condition. The relaxation kernel $k$ is monotone decreasing and $k(\infty)=1$.
We study  blow-up of solutions when the source is  stronger than  dissipations, i.e., $p> \max\{m,\sqrt{k(0)}\}$, under two different scenarios: first, the total energy is negative, and the second, the total energy is positive with sufficiently large quadratic energy. This manuscript is a follow-up work of the paper \cite{visco-Guo-Ramm} in which Hadamard well-posedness of this equation has been established in the finite energy space. The model under consideration features a \emph{supercritical}  source and a linear memory that accounts for the full past history as time goes to $-\infty$, which is distinct from other relevant models studied in the literature which usually involve subcritical sources and a finite-time memory.
\end{abstract}

\section {Introduction}\label{S1}
\subsection{The model and  literature overview}
Viscoelastic materials demonstrate properties between those of elastic materials and viscous fluid. In the nineteenth century, Boltzmann \cite{Boltz1} realized that the behavior of these materials should be modeled through constitutive relations that involve long but fading memory. In particular, Boltzmann  initiated the classical linear theory of viscoelasticity. As a consequence of the widespread use of polymers and other modern materials which exhibit stress relaxation, the theory of viscoelasticity has provided important applications in materials science and engineering. Please see \cite{Col1} (and references therein) for the fundamental modeling development of linear viscoelasticity. We also refer the reader to the monographs \cite{Fab2,Renardy} for surveys regarding the mathematical aspect of the theory of viscoelasticity. In addition, the literature is quite  rich in various results on well-posedness and asymptotic stability of hyperbolic PDEs and conservation laws with  \emph{memory} terms of Boltzmann type, see for instance \cite{Cav5,Chen-Christ,Daf2,Daf1,Daf3,Daf4,DPZ,Guesmia,LMM,Pata}  and the references therein.

In this manuscript, we investigate the following nonlinear hyperbolic equation of viscoelasticity:
\begin{align} \label{1.1}
\begin{cases}
u_{tt}- k(0) \Delta u -  \int_0^{\infty} k'(s) \Delta u(t-s) ds + |u_t|^{m-1}u_t=|u|^{p-1}u,\text{\;in\;} \O \times (0,T)\\
u(x,t)=0, \quad \text{ on }  \Gamma \times (-\infty,T) \\
u(x,t)=u_0(x,t), \quad \text{ in } \O \times (-\infty,0],
\end{cases}
\end{align}
where the unknown $u(x,t)$ is an $\reals$-valued function defined on $\Omega \times (-\infty,T)$, and $\O\subset \reals^3$ is a bounded domain (open and connected) with smooth boundary $\G$. Our results extend easily to bounded domains in $\mathbb R^n$, by accounting for the corresponding Sobolev embedding, and accordingly adjusting the conditions imposed on the parameters.
 The system (\ref{1.1}) models the wave propagation in viscoelastic material under the influence of frictional type of damping as well as energy-amplifying sources. Here, $|u_t|^{m-1}u_t$ ($m\geq 1$) represents a nonlinear damping which dissipates energy and drives the system toward stability, while $|u|^{p-1}u$ ($1\leq p<6$) represents a nonlinear source of \emph{supercritical} growth rate which models an external force that amplifies energy and drives the system to possible instability. The memory integral $\int_0^{\infty} k'(s) \Delta u(t-s) ds$ of the Boltzmann type quantifies the viscous resistance and provides a weak form of energy dissipation by assuming that the relaxation kernel satisfies: $k'(s)< 0 \text{\;\;for all\;\;} s>0$ and $k(\infty)=1$. It also accounts for the full past history as time goes to $-\infty$, as opposed to the finite-memory models where the history is taken only over the interval $[0,t]$.

Nonlinear wave equations under the influence of damping and sources have been attracting considerable attention in the research field of analysis of nonlinear PDEs. In \cite{G-T}, Georgiev and Todorova considered a nonlinear wave equation with damping and sources:
\begin{align}  \label{wave}
u_{tt}- \Delta u + |u_t|^{m-1}u_t = |u|^{p-1}u,       \text{\;\;\;in } \O \times (0,T),
\end{align}
under a Dirichl\'et boundary condition, where $1<p\leq 3$ in 3D. They showed  that equation (\ref{wave}) is globally well-posed in the finite energy space $H^1_0(\O) \times L^2(\O)$ in the case $1<p\leq m$. In addition,  a blow-up result was obtained in \cite{G-T} provided the initial energy is sufficiently negative and $1<m<p$.
The related wave equations with nonlinear \emph{boundary} damping--source interactions have been studied in \cite{Viti} by Vitillaro, and in \cite{CCL} by Cavalcanti, Domingos Cavalcanti and Lasiecka. Also we would like to mention the work \cite{BLR,BLR2,BLR3} by Barbu, Lasiecka and Rammaha, where they investigated wave equations with \emph{degenerate} damping and source terms,  where  the prototype equation of this class is
\begin{align}   \label{dege}
u_{tt}-\Delta u + |u|^k |u_t|^{m-1}u_t = |u|^{p-1}u,      \text{\;\;\;in } \O \times (0,T),
\end{align}
with $u=0$ on the boundary. In (\ref{dege}) the degenerate damping $|u|^k |u_t|^{m-1}u_t$ models  friction modulated by strain. Well-posedness of various types of solutions, such as generalized solutions, weak solutions and strong solutions, to the system (\ref{dege}) has been established with certain assumptions on parameters $k$, $m$, $p$, and moreover,  solutions are global if $p\leq k+m$, and blow up if $p>k+m$ and  initial energy is negative (see \cite{BLR,BLR2, BLR3}).
We also refer the reader to the papers \cite{RS1,RS2} for a study of a system of wave equations with coupled degenerate damping and coupled nonlinear sources.
For more works on nonlinear wave and hyperbolic equations with damping--source interactions, we mention
\cite{ACC,Agre-Ramm,ACCRT,CCP,Gazz, Messaoudi1, Messaoudi2, Levine-Serrin,Pei1,Pei2,Pitts-Ramm,Ramm-Strei} and the references therein.

 Our work in this article follows the recent trend in studying nonlinear wave equations with \emph{supercritical} sources. We say a typical source term $|u|^{p-1}u$ \emph{supercritical} if $3<p<6$ (in 3D), and under such scenario, the mapping  $u \mapsto |u|^{p-1}u$ is not a locally Lipschitz mapping from $H^1_0(\O)$ into $L^2(\O)$, and therefore, the classical fixed-point technique used in \cite{G-T} is not directly applicable to establish the well-posedness in finite energy space. A breakthrough was made in a series of papers \cite{BL3,BL2,BL1} by Bociu and Lasiecka in studying a nonlinear wave equation with damping and \emph{supercritical} sources acting in the interior of the domain and on the boundary,  where a delicate analysis within the framework of the theory of semi-groups and monotone operators \cite{Barbu3,Sh} was used to establish local well-posedness of weak solutions. Please refer to the papers \cite{BGRT,BRT,BRT-IMACS,GR,GR1,GR2,Pei3,RTW,RW} for more work on various hyperbolic PDEs under the influence of \emph{supercritical} nonlinearities. In particular, the local well-posedness of the viscoelastic wave equation (\ref{1.1}) with \emph{supercritical} sources and damping was established in \cite{visco-Guo-Ramm} by adopting the approach from \cite{BL1,GR}, and in addition, the extension to global solutions was studied in the case that the damping dominates the source. This manuscript is a follow-up work of \cite{visco-Guo-Ramm}, and we investigate the conditions under which the system (\ref{1.1}) blows up in finite time. We note here that equation (\ref{1.1}) is equipped with two types of dissipation: the linear memory $-\int_0^{\infty} k'(s) \Delta u(t-s) ds$ with $k'(s)<0$ as well as the frictional damping $|u_t|^{m-1}u_t$, but only one source term $|u|^{p-1}u$, therefore, it would be quite interesting to explore the mechanism of how the source surpasses the two dissipations leading to a blow-up of the system.

We would like to emphasize that our study of the blow-up of (\ref{1.1}) is under two different situations: negative initial energy as well as the positive initial energy. Also, it is important to notice that, in our model (\ref{1.1}), the source is \emph{supercritical} and the linear memory accounts for the full past history as time goes to $-\infty$, which is distinct from other relevant models studied in the literature which usually consider only subcritical sources ($1\leq p\leq 3$ in 3D) and a finite-time memory. In the proof of our results, we carefully justify all the formal calculations, and so  our work is fully rigorous. As a matter of fact, the analysis is quite involved and subtle, and in particular in the case of positive initial energy, due to the presence of the memory term in the equation, it is nontrivial to find an assumption on the upper bound of the initial energy.

\smallskip

\subsection{Review of well-posedness results in \cite{visco-Guo-Ramm}} In this subsection, we shall review the well-posedness results for (\ref{1.1}) obtained in our previous paper \cite{visco-Guo-Ramm} coauthored with Titi and Toundykov.
Throughout the paper, we use the notation
 $$\mathbb R^+=[0,\infty) \text{\;\;\;and\;\;\;} \mathbb R^-=(-\infty,0].$$
For the purpose of defining the proper function space for the initial data, we set $$\mu(s)=-k'(s).$$ Thus $\mu: \reals^+\longrightarrow \reals^+$, and in Assumption \ref{ass} below precise assumptions on $\mu$ will be imposed. We assume that the initial datum is a function $u_0(x,t)$ defined for negative times $t\in \mathbb R^-$ and in particular $u_0(x,t): \Omega\times \mathbb R^- \rightarrow \mathbb R$ belongs to a weighted Hilbert space $L^2_{\mu}(\mathbb R^-, H^1_0(\O))$, i.e.,
$$\|u_0\|^2_{L^2_{\mu}(\mathbb R^-, H^1_0(\O))}=\int_0^{\infty} \int_{\O} |\nabla u_0(x,-t)|^2 dx \mu(t) dt<\infty,$$
and $\partial_t u_0\in L^2_{\mu}(\reals^-,L^2(\O))$, that is,
\begin{align*}
\|\partial_t u_0\|^2_{L^2_{\mu}(\mathbb R^-, L^2(\O))}=\int_0^{\infty} \int_{\O} |\partial_t u_0(x,-t)|^2 dx \mu(t) dt<\infty.
\end{align*}
Also, the standard $L^s(\O)$-norm will be denoted by:
\begin{align*}
\norm{u}_s=\norm{u}_{L^s(\O)}.
\end{align*}

The following assumptions will be imposed throughout the manuscript.
\begin{assumption} \label{ass}\leavevmode
\begin{itemize}
\item ~~ $m\geq 1$, $1\leq p<6$, $p\frac{m+1}{m}<6$;
\item ~~ $k\in C^2(\reals^+)$ such that $k'(s)<0$ for all $s>0$ and $k(\infty)=1$;
\item ~~ $\mu(s)=-k'(s)$ such that $\mu \in C^1(\reals^+)\cap L^1(\reals^+)$ and $\mu'(s)\leq 0$ for all $s>0$, and $\mu(\infty)=0$;
\item ~~ $u_0(x,t)\in L^2_{\mu}(\reals^-,H^1_0(\O))$,\, $\partial_t u_0(x,t)\in L^2_{\mu}(\reals^-,L^2(\O))$ such that\\ $u_0: \mathbb R^- \rightarrow H^1_0(\O)$ and
$\partial_t u_0(x,t): \mathbb R^- \rightarrow L^2(\O)$ are weakly continuous at $t=0$. In addition, for all $t\leq 0$, $u_0(x,t)=0$ on $\G$.
\end{itemize}
\end{assumption}

We begin with giving the definition of a weak solution of (\ref{1.1}).
\begin{definition} \label{def-weak}
A function $u(x,t)$ is said to be a \emph{weak solution} of (\ref{1.1}) defined on the time interval $(-\infty,T]$
provided $u\in C([0,T];H^1_0(\O))$ such that $u_t\in C([0,T];L^2(\O))\cap L^{m+1}(\O \times (0,T))$ with:
\begin{itemize}
\item $u(x,t)=u_0(x,t)\in L^2_{\mu}(\reals^-,H^1_0(\O))$ for $t\leq 0$;
\item The following variational identity holds for all $t\in [0,T]$  and all test functions $\phi \in \mathscr{F}$:
\begin{align} \label{weak}
&\int_{\O}u_t(t)\phi(t) dx -  \int_{\O}u_t(0)\phi(0) dx   -\int_0^t \int_{\O}u_t(\t)\phi_t(\t)dx d\t\notag\\
&+k(0)\int_0^t \int_{\O} \nabla u(\t) \cdot \nabla \phi(\t) dx d\t
+\int_0^t \int_0^{\infty} \int_{\O} \nabla u(\t-s) \cdot \nabla \phi(\t) dx k'(s) ds d\t\notag\\
&+\int_0^t \int_{\O} |u_t(\tau)|^{m-1}u_t(\tau)\phi(\t) dx d\t
=\int_0^t \int_{\O} |u(\tau)|^{p-1}u(\tau) \phi(\t) dx d\t,
\end{align}
where $$ \mathscr{F} = \Big\{ \phi: \,\, \phi \in C([0,T];H^1_0(\O))\cap L^{m+1}(\O \times (0,T)) \text{  with  } \phi_t\in C([0,T];L^2(\O)) \Big\}.$$
\end{itemize}
\end{definition}

For the reader's convenience we summarize the main results obtained in \cite{visco-Guo-Ramm} which are relevant to the work in this paper.

\begin{theorem}[{\bf Short-time existence and uniqueness} \cite{visco-Guo-Ramm}] \label{thm-exist}
Assume the validity of the Assumption \ref{ass}, then there exists a local (in time) weak solution $u$ to (\ref{1.1}) defined on the time interval $(-\infty,T]$ for some $T>0$ depending on the initial quadratic energy $\mathscr E(0)$. Furthermore, the following energy identity holds:
\begin{align} \label{EI-0}
&\mathscr E(t)+\int_0^t \int_{\O} |u_t|^{m+1} dx d\t-\frac{1}{2}\int_0^t \int_0^{\infty} \norm{\nabla w(\t, s)}_2^2 \mu'(s)ds d\t \notag\\
&=\mathscr E(0)+\int_0^t \int_{\O} |u|^{p-1}u u_t dx d\t,
\end{align}
where the history function $w(x,\tau,s)$ and the quadratic energy $\mathscr E(t)$ are respectively defined by:
\begin{align}\label{1.6}
\begin{cases}
w(x,\t,s)=u(x,\t)-u(x,\t-s),\vspace{.1in}\\
\mathscr E(t)=\frac{1}{2}\left(\norm{u_t(t)}_2^2+\norm{\nabla u(t)}_2^2+\int_0^{\infty}\norm{\nabla w(t,s)}_2^2 \mu(s)ds \right).
\end{cases}
\end{align}
 If in addition we assume $u_0(0)\in L^{\frac{3(p-1)}{2}}(\O)$, then weak solutions of (\ref{1.1}) are unique.
\end{theorem}

\vspace{0.05 in}

\begin{remark}
If we assume that $u(t)\in L^{p+1}(\O)$ for $t$ belonging to the lifespan $(-\infty,T]$ of the local solution (or instead assume $p\leq 5$), then the \emph{total energy} $E(t)$ of the system (\ref{1.1}) is defined by
\begin{align}  \label{to-ene}
E(t)&=\mathscr E(t) -  \frac{1}{p+1}\|u(t)\|_{p+1}^{p+1} \notag \\
&=\frac{1}{2}\left(\norm{u_t(t)}_2^2+\norm{\nabla u(t)}_2^2+\int_0^{\infty}\norm{\nabla w(t,s)}_2^2 \mu(s)ds \right)- \frac{1}{p+1}\|u(t)\|_{p+1}^{p+1}.
\end{align}
It is readily seen that, in terms of the total energy $E(t)$, the energy identity (\ref{EI-0}) can be written as
\begin{align} \label{EI-1}
E(t)+\int_0^t \int_{\O} |u_t|^{m+1} dx d\t-\frac{1}{2}\int_0^t \int_0^{\infty} \norm{\nabla w(\t, s)}_2^2 \mu'(s)ds d\t =E(0).
\end{align}
\end{remark}

The next result states that weak solutions of (\ref{1.1}) depend continuously on the initial data.
\begin{theorem}[{\bf Continuous dependence on initial data} \cite{visco-Guo-Ramm}]  \label{thm-cont}
In addition to the Assumption \ref{ass}, assume that  $u_0(0)\in L^{\frac{3(p-1)}{2}}(\O)$.
If $u_0^n\in L^2_{\mu}(\reals^-, H^1_0(\O))$ is a sequence of initial data such that $u_0^n\longrightarrow u_0$ in $L^2_{\mu}(\reals^-,H^1_0(\O))$ with
$u^n_0(0)\longrightarrow u_0(0)$ in $H^1_0(\O)$ and in $L^{\frac{3(p-1)}{2}}(\O)$, $\frac{d}{dt}u^n_0(0)\longrightarrow \frac{d}{dt}u_0(0)$ in $L^2(\O)$, then the corresponding weak solutions $u_n$ and $u$ of (\ref{1.1}) satisfy
\begin{align*}
u_n\longrightarrow u \text{\;\;in\;\;}  C([0,T];H^1_0(\O)) \text{\;\;\;and\;\;\;}
(u_n)_t\longrightarrow u_t  \text{\;\;in\;\;}  C([0,T];L^2(\O)).
\end{align*}
\end{theorem}

The following result states: if the damping dominates the source term, then the solution is global.
\begin{theorem}[{\bf Global existence} \cite{visco-Guo-Ramm}] \label{thm-global}
In addition to Assumption \ref{ass}, further assume $u_0(0) \in L^{p+1}(\O)$. If $m\geq p$, then the weak solution of (\ref{1.1}) is global.
\end{theorem}

\smallskip

\subsection{Main results}
The main results of the paper consist of two theorems concerning the finite-time blow-up of the system (\ref{1.1}).  We prove these results for negative and positive initial energy when the source term is more dominant than the frictional damping as well as the dissipation from the delay.

Our first blow-up result deals with the case when the initial total energy $E(0)$ is negative. Specifically, we have the following theorem.
\begin{theorem}[{\bf Blow-up of solutions with negative initial energy}]   \label{thm-blow}
Assume the validity of the Assumption \ref{ass} and $E(0)<0$. If $p>\max\{m,\sqrt{k(0)}\}$, then the weak solution $u$ of (\ref{1.1}) blows up in finite time. More precisely, $\limsup_{t\rightarrow T_{max}^-}\|\nabla u(t)\|_2=\infty$, for some $T_{\max}\in (0,\infty)$.
\end{theorem}
\begin{remark}  \label{rem}
Although the well-posedness results, Theorems \ref{thm-exist} and \ref{thm-cont}, allow the growth rate $p$ of the source term to take any value in $[1,6)$, nevertheless the assumptions in Theorem \ref{thm-blow} force the restriction $p<5$. Indeed, if we combine the assumptions $p>m$ and $p\frac{m+1}{m}<6$ from Assumption \ref{ass}, we find that
$6>p(1+\frac{1}{m})>p(1+\frac{1}{p})=p+1$, which implies that $p<5$.
\end{remark}

Our second blow-up result is concerned with the case that the initial total energy is nonnegative. For a given $p\in (1,5]$, let $\gamma>0$ be the best constant for the Sobolev inequality $\|u\|_{p+1} \leq \gamma \|\nabla u\|_2$ for all $u\in H^1_0(\O)$, i.e.,
\begin{align}  \label{em}
\gamma^{-1}=\inf \left\{\norm{\nabla u}_2: u\in H^1_0(\O),  \|u\|_{p+1}=1 \right\}.
\end{align}
Then, we have the following result.
\begin{theorem}[{\bf Blow-up of solutions with positive initial energy}]   \label{thm-blow2}
 In addition to the validity of the Assumption \ref{ass}, we assume that $p>\max\{m,\sqrt{k(0)}\}$. Further assume that $\mathscr E(0) > y_0:=\frac{1}{2} \gamma^{-\frac{2(p+1)}{p-1}}$ and
 \begin{align} \label{def-M}
0\leq E(0)<M:=(\sqrt{k(0)}+1)^{\frac{2}{p-1}}  (2\gamma^2)^{-\frac{p+1}{p-1}}\left(\frac{p-\sqrt{k(0)}}{p+1}\right).
\end{align}
Then, the weak solution $u$ of the system (\ref{1.1}) blows up in finite time. More precisely,
$\limsup_{t\rightarrow T^-_{max}}\|\nabla u(t)\|_2 = \infty$, for some $T_{\max}\in (0,\infty)$.
\end{theorem}

The remainder of the manuscript is organized as follows. Section \ref{S2} is devoted to the proof of Theorem \ref{thm-blow}, where we show blow-up of the weak solution to (\ref{1.1}) when the initial total energy is \emph{negative} and the source dominates the frictional damping and the dissipation due to the memory term. In Section \ref{S3}, we present the proof of Theorem \ref{thm-blow2} which contains a finite-time blow-up result in the case of \emph{positive} initial total energy and with sufficiently large quadratic energy.

\bigskip

\section{Proof of Theorem \ref{thm-blow}} \label{S2}
This section is devoted to proving the blow-up of weak solutions to the viscoelastic wave equation (\ref{1.1}) when the total energy is negative.
In particular, we shall present a rigorous proof of Theorem \ref{thm-blow}, which states that, if the initial energy $E(0)$ is negative and the source dominates dissipation in the system, i.e. $p>\max\{m,\sqrt{k(0)}\}$, then the weak solution of (\ref{1.1}) blows up in finite time.
\begin{proof}
 Let $u$ be a weak solution of the system (\ref{1.1}) in the sense of Definition \ref{def-weak}. We define the life span $T_{max}$ of the solution $u$ to be the supremum of all $T>0$ such that $u$ is a solution of (\ref{1.1}) on $(-\infty,T]$. We aim to show that $T_{max}$ is necessarily finite, that is, $u$ blows up in finite time.

The main idea of the proof is due to \cite{G-T} (see also \cite{Levine-Serrin,Messaoudi1}).
One major contribution of the paper \cite{G-T} was the choice of a special Liapunov's function for the purpose of proving the blow-up result.
Indeed, we put $G(t)=-E(t)$ and $N(t)=\frac{1}{2}\|u(t)\|_2^2$. We aim to show
\begin{align} \label{def-Y}
Y(t)=G(t)^{1-\alpha}+\epsilon N'(t)
\end{align}
blows up in finite time, for some $\alpha\in (0,1)$ and $\epsilon>0$, which will be selected later.

We calculate
\begin{align*}
Y'(t)=(1-\alpha)G(t)^{-\alpha}G'(t)+\epsilon N''(t).
\end{align*}
By the definition of weak solutions, i.e, Definition \ref{def-weak}, we find the regularity of weak solutions: $u\in  C([0,T];H^1_0(\O))$ such that $u_t\in C([0,T];L^2(\O))\cap L^{m+1}(\O \times (0,T))$ for $0<T<T_{max}$. Clearly one has
$$N'(t)=\int_{\O} uu_t dx.$$
Also, \emph{formally} we have $N''(t)=\frac{d}{dt}\int_{\O}  u u_t dx = \|u_t\|_2^2 + \int_{\O} u u_{tt} dx$. However, because of the lack of the regularity of $u_{tt}$, such formal calculation is not legitimate. In order to bypass this obstacle, we resort to the variational identity (\ref{weak}) and we would like to use $u$ as a test function in place of $\phi$ in (\ref{weak}) in order to obtain an identity for $N'(t)$. To proceed in this direction, we shall check that whether $u$ belongs to the admissible set $\mathscr F$ of test functions $\phi$. By the regularities of weak solutions, we know that $u\in C([0,T];H^1_0(\O))$ and $u_t \in C([0,T];L^2(\O))$ for $0<T<T_{\max}$, and this immediately implies that $u\in L^{m+1}(\O \times (0,T))$ since $H^1_0(\O) \hookrightarrow L^{m+1}(\O)$ due to $m<p<5$ by Remark \ref{rem}. Hence, the solution $u$ enjoys the regularity restrictions imposed on the test functions in $\mathscr F$, as stated in Definition \ref{def-weak}. As a result, we may replace $\phi$ by $u$ in the variational identity (\ref{weak}) to obtain
\begin{align}  \label{deN}
N'(t) &=\int_{\O}u_t(0)u(0) dx   +\int_0^t \|u_t(\tau)\|_2^2 d\t \notag\\
&\hspace{0.4 in}-k(0)\int_0^t  \|\nabla u(\t)\|^2_2  d\t
-\int_0^t \int_0^{\infty} \int_{\O} \nabla u(\t-s) \cdot \nabla u(\t) dx k'(s) ds d\t\notag\\
&\hspace{0.4 in}-\int_0^t \int_{\O} |u_t(\tau)|^{m-1}u_t(\tau)u(\t) dx d\t
+\int_0^t \|u(\tau)\|^{p+1}_{p+1}  d\t \notag\\
&=\int_{\O}u_t(0)u(0) dx   +\int_0^t \|u_t(\tau)\|_2^2 d\t \notag\\
&\hspace{0.4 in}-\int_0^t  \|\nabla u(\t)\|^2_2  d\t + \int_0^t \int_0^{\infty} \int_{\O} \nabla w(\t,s) \cdot \nabla u(\t) dx k'(s) ds d\t\notag\\
&\hspace{0.4 in}-\int_0^t \int_{\O} |u_t(\tau)|^{m-1}u_t(\tau)u(\t) dx d\t
+\int_0^t \|u(\tau)\|^{p+1}_{p+1}  d\t,
\end{align}
for all $t\in [0,T_{max})$, where we have used $w(x,t,s)=u(x,t)-u(t-s)$ as well as the assumption $k(\infty)=1$.

In order to differentiate $N'(t)$, we shall verify that $N'(t)$ is absolutely continuous on any closed subsegment of $[0,T_{max})$. By the assumptions $k'(s)<0$, $\mu(s)=-k'(s)>0$ and $k(\infty)=1$, one has
\begin{align}
&\int_0^t \left|\int_0^{\infty} \int_{\O} \nabla w(\tau,s) \cdot \nabla u(\tau) dx k'(s) ds \right| d\tau \notag\\
&\leq  \int_0^t \int_0^{\infty}  \|\nabla w(\tau,s)\|_2^2 \mu(s) ds d\tau
- \int_0^t \int_0^{\infty}  \|\nabla u(\tau) \|_2^2  k'(s) ds d\tau     \notag\\
&\leq   2\int_0^t \mathscr E(\tau)  d\tau
+(k(0)-1)\int_0^t \|\nabla u(\tau)\|_2^2 d\tau<\infty,
\end{align}
for all $t\in [0,T_{max})$, where we have used the fact that $u\in C([0,t];H^1_0(\O))$ and $\mathscr E(t)$ is continuous due to the energy identity (\ref{EI-0}). Also, by H\"older's and Young's inequalities, one has
\begin{align*}
\int_0^t \left|\int_{\O} |u_t(\tau)|^{m-1}u_t(\tau)u(\t) dx\right| d\t
&\leq \int_0^t \|u_t\|_{m+1}^m \|u\|_{m+1} d\t   \notag\\
&\leq C\int_0^t \|u_t\|_{m+1}^{m+1} d\tau + C\int_0^t \|u\|_{m+1}^{m+1} d\tau < \infty,
\end{align*}
for all $t\in [0,T_{max})$, since $u_t \in L^{m+1}(\Omega \times (0,t))$ and $u\in C([0,t];H^1_0(\O))$ as well as the
imbedding $H^1(\O) \hookrightarrow L^{m+1}(\O)$ due to the fact $m<p<5$ from Remark \ref{rem}.
Therefore,  $N'(t)$ is absolutely continuous on any closed subsegment of $[0,T_{max})$. Thus, we may differentiate again in (\ref{deN}) to obtain:
\begin{align}   \label{doub}
N''(t)= &\|u_t\|_2^2  - \|\nabla u\|_2^2 + \int_0^{\infty} k'(s) \int_{\O} \nabla u(t) \cdot \nabla w(t,s) dx ds   \notag\\
&- \int_{\O}  |u_t|^{m-1} u_t u dx + \|u\|_{p+1}^{p+1},  \text{\;\;\;\;\;for all\;\;} t\in  [0,T_{max}).
\end{align}

The next step is to find an appropriate lower bound of right-hand side of (\ref{doub}). Indeed, by using the Cauchy-Schwarz and Young's inequalities, and the assumption $\mu(s)=-k'(s)>0$, one has
\begin{align} \label{ks}
&\left|\int_0^{\infty} k'(s) \int_{\O} \nabla u(t) \cdot \nabla w(t,s) dx ds\right|     \notag\\
&\leq  \int_0^{\infty} (-k'(s)) \left( \frac{1}{4\delta} \|\nabla u(t)\|_2^2  + \delta \|\nabla w(t,s)\|^2_2  \right) ds  \notag\\
&\leq \frac{k(0)-1}{4\delta} \|\nabla u\|_2^2  + \delta \int_0^{\infty}   \|\nabla w(t,s)\|^2_2  \mu(s)  ds,
\end{align}
for some $\delta>0$ whose value will be selected later.  Also, by applying H\"older's and Young's inequalities, and using the assumption that the source is stronger than the damping, i.e. $p>m$, we obtain
\begin{align} \label{ut}
\left|\int_{\O} u |u_t|^{m-1} u_t dx \right|\leq \int_{\O} |u||u_t|^m dx \leq \|u\|_{m+1} \|u_t\|_{m+1}^m
\leq C\|u\|_{p+1} \|u_t\|_{m+1}^m.
\end{align}
Since $G(t)=-E(t)$ and $\mu'(s)\leq 0$, (\ref{EI-1}) implies
\begin{align}  \label{Gprime}
G'(t)= \|u_t\|_{m+1}^{m+1}  -\frac{1}{2}  \int_0^{\infty} \norm{\nabla w(t, s)}_2^2 \mu'(s)ds  \geq \|u_t\|_{m+1}^{m+1} \geq 0.
\end{align}
Thus, $G(t)$ is nondecreasing for $t\in [0,T_{max})$. Moreover, by (\ref{to-ene}),
\begin{align} \label{GG}
G(t)=-E(t)= -\mathscr E(t) +  \frac{1}{p+1}\|u(t)\|_{p+1}^{p+1} \leq \frac{1}{p+1} \|u(t)\|_{p+1}^{p+1}.
\end{align}

Now, by applying (\ref{GG}) to inequality (\ref{ut}) and invoking the assumption $p>m$, we deduce that
\begin{align}   \label{ut-0}
\left|\int_{\O} u |u_t|^{m-1} u_t dx \right|
&\leq   C \|u\|^{1-\frac{p+1}{m+1}}_{p+1}   \left(\|u\|_{p+1}^{\frac{p+1}{m+1}} \|u_t\|_{m+1}^m\right) \notag\\
&\leq   C G(t)^{\frac{1}{p+1}-\frac{1}{m+1}}  \left(\|u\|_{p+1}^{\frac{p+1}{m+1}} \|u_t\|_{m+1}^m\right)    \notag\\
&\leq  \lambda G(t)^{\frac{1}{p+1}-\frac{1}{m+1}} \|u\|_{p+1}^{p+1} +  C_{\lambda}  G(t)^{\frac{1}{p+1}-\frac{1}{m+1}}  \|u_t\|_{m+1}^{m+1},
\end{align}
where we have used the Young's inequality and the value of the positive number $\lambda$ will be determined later. By selecting $$0<\alpha<\frac{1}{m+1}-\frac{1}{p+1}$$ and using (\ref{Gprime}), we obtain
\begin{align}  \label{ut-1}
\left|\int_{\O} u |u_t|^{m-1} u_t dx \right|
&\leq \lambda G(t)^{\frac{1}{p+1}-\frac{1}{m+1}} \|u\|_{p+1}^{p+1} +  C_{\lambda}  G(t)^{\frac{1}{p+1}-\frac{1}{m+1}+\alpha} G(t)^{-\alpha}  \|u_t\|_{m+1}^{m+1} \notag\\
&\leq \lambda G(0)^{\frac{1}{p+1}-\frac{1}{m+1}} \|u\|_{p+1}^{p+1} +  C_{\lambda}  G(0)^{\frac{1}{p+1}-\frac{1}{m+1}+\alpha} G(t)^{-\alpha}  G'(t).
\end{align}

By employing (\ref{ks}) and (\ref{ut-1}), we obtain from (\ref{doub}) that
\begin{align}  \label{doub-1}
N''(t) \geq &\|u_t\|_2^2 - \left(\frac{k(0)-1}{4\delta}+1\right) \|\nabla u\|_2^2 - \delta \int_0^{\infty} \norm{\nabla w(t,s)}^2_2 \mu(s)ds \notag\\
&+(1-\lambda G(0)^{\frac{1}{p+1}-\frac{1}{m+1}} )\|u\|_{p+1}^{p+1}-C_{\lambda}  G(0)^{\frac{1}{p+1}-\frac{1}{m+1}+\alpha} G(t)^{-\alpha}  G'(t),
\end{align}
for $t\in [0,T_{max})$. Since $G(t)=-E(t)$, we obtain from (\ref{to-ene}) that
\begin{align}  \label{EI-2}
\int_0^{\infty}\norm{\nabla w(t,s)}_2^2 \mu(s)ds = -2 G(t)- \norm{u_t(t)}_2^2 - \|\nabla u\|_2^2 +\frac{2}{p+1}\|u(t)\|_{p+1}^{p+1}.
\end{align}
By substituting (\ref{EI-2}) into (\ref{doub-1}), one has
\begin{align} \label{doub-2}
&N''(t) \geq  (1+\delta)\|u_t\|_2^2+2\delta G(t) + \left(\delta -\frac{k(0)-1}{4\delta}-1  \right)  \norm{\nabla u}_2^2 \notag\\
&+\left(1-\frac{2\delta}{p+1}-\lambda G(0)^{\frac{1}{p+1}-\frac{1}{m+1}} \right)\|u\|_{p+1}^{p+1} -C_{\lambda}  G(0)^{\frac{1}{p+1}-\frac{1}{m+1}+\alpha} G(t)^{-\alpha}  G'(t),
\end{align}
for $t\in [0,T_{max})$. We intend to select $\delta>0$ such that
\begin{align*}
\delta -\frac{k(0)-1}{4\delta}-1  \geq 0    \text{\;\;\;and\;\;\;}   1-\frac{2\delta}{p+1}>0.
\end{align*}
These two restrictions imply that
\begin{align*}
\frac{\sqrt{k(0)}+1}{2} \leq \delta  < \frac{p+1}{2},
\end{align*}
which is valid since $p>\sqrt{k(0)}$ by the assumption of the theorem. In the following we choose
\begin{align*}
\delta = \frac{\sqrt{k(0)}+1}{2},
\end{align*}
and select $\lambda>0$ such that
\begin{align*}
\lambda G(0)^{\frac{1}{p+1}-\frac{1}{m+1}}=\frac{1}{2}-\frac{\delta}{p+1}
=\frac{p-\sqrt{k(0)}}{2(p+1)},
\end{align*}
then  inequality (\ref{doub-2}) can be reduced to
\begin{align*}
&N''(t) \geq \frac{1}{2}\left(\sqrt{k(0)}+3\right)\|u_t\|_2^2+\left(\sqrt{k(0)}+1\right)G(t)  \notag\\
&+\frac{p-\sqrt{k(0)}}{2(p+1)}\|u\|_{p+1}^{p+1} -C_{\lambda}  G(0)^{\frac{1}{p+1}-\frac{1}{m+1}+\alpha} G(t)^{-\alpha}  G'(t),
\end{align*}
for all $t\in [0,T_{max})$. Now, since $Y'(t)=(1-\alpha)G(t)^{-\alpha}G'(t)+\epsilon N''(t)$, if we select $\epsilon>0$ small enough so that  $$\epsilon C_{\lambda}G(0)^{\frac{1}{p+1}-\frac{1}{m+1}+\alpha} \leq 1-\alpha.$$
Thus, one has
\begin{align}  \label{Y-0}
Y'(t) \geq &  \frac{\epsilon}{2}\left(\sqrt{k(0)}+3\right)\|u_t\|_2^2+\epsilon\left(\sqrt{k(0)}+1\right)G(t)+\frac{\epsilon\left[p-\sqrt{k(0)}\right]}{2(p+1)} \|u\|_{p+1}^{p+1},
\end{align}
for $t\in [0,T_{max})$.

Recall that $G(0)=-E(0)>0$, and since $G(t)$ is nondecreasing by (\ref{Gprime}), it follows that $G(t)>0$ for $t\in [0,T_{max})$.
Thanks to (\ref{Y-0}), we have $Y'(t)>0$, i.e. $Y(t)$ is monotone increasing for $t\in [0,T_{max})$.
Note that $Y(0)=G(0)^{1-\alpha}+\epsilon N'(0)$. If in case $N'(0)<0$, in order to make sure that $Y(0)>0$, we shall impose an extra restriction on $\epsilon$:
\begin{align*}
0<\epsilon \leq -\frac{G(0)^{1-\alpha}}{2N'(0)}.
\end{align*}
As a result,
\begin{align}  \label{YG}
Y(t)\geq Y(0) \geq \frac{1}{2}G(0)^{1-\alpha}>0 \text{\;\;\;for all\;\;\;} t\in [0,T_{max}).
\end{align}

Recall the assumption $p>m\geq 1$   and our choice of $\a$, namely,  $0<\alpha<\frac{1}{m+1}-\frac{1}{p+1}$.  Thus, $\a<\frac{1}{2}$, and in particular, $1<\frac{1}{1-\alpha}<2$. We aim to show that
\begin{align}   \label{Y-1}
Y'(t) \geq \epsilon^{1+\sigma} C(k(0),p)  Y(t)^{\frac{1}{1-\alpha}},  \text{\;\;\;for\;\;} t\in [0,T_{max}),
\end{align}
where $\sigma=1-\frac{2}{(1-2\alpha)(p+1)}$. If (\ref{Y-1}) is valid, then we will have $Y(t)$ blows up in finite time, due to the fact that $Y(0)>0$ and $\frac{1}{1-\alpha}>1$.

Since $Y(t)=G(t)^{1-\alpha}+\epsilon N'(t)$, if we let $\epsilon\leq 1$, it follows that
\begin{align}  \label{Y-2}
Y(t)^{\frac{1}{1-\alpha}} \leq C\left(G(t)+|N'(t)|^{\frac{1}{1-\alpha}}\right), \text{\;\;\;for\;\;} t\in [0,T_{max}).
\end{align}
Since $N'(t)=\int_{\O}uu_t dx$, then by the Cauchy-Schwarz and Young's inequalities, we have
\begin{align}  \label{doub-3}
|N'(t)|^{\frac{1}{1-\alpha}}
\leq \|u_t\|_2^{\frac{1}{1-\alpha}} \|u\|_2^{\frac{1}{1-\alpha}}
\leq C \|u_t\|_2^{\frac{1}{1-\alpha}} \|u\|_{p+1}^{\frac{1}{1-\alpha}}
\leq C\left(\|u_t\|_2^2 + \|u\|_{p+1}^{\frac{2}{1-2\alpha}}\right).
\end{align}
Notice that
\begin{align}   \label{p-1}
\|u\|_{p+1}^{\frac{2}{1-2\alpha}}=\left(\|u\|_{p+1}^{p+1}\right)^{\frac{2}{(1-2\alpha)(p+1)}}
=\left(\|u\|_{p+1}^{p+1}\right)^{\frac{2}{(1-2\alpha)(p+1)}-1} \|u\|_{p+1}^{p+1}.
\end{align}
Now we impose an extra restriction on $\alpha$:
$$0<\alpha<\frac{p-1}{2(p+1)},$$
 then $\sigma=1-\frac{2}{(1-2\alpha)(p+1)}>0$. By virtue of (\ref{GG}) and the fact that $G(t)$ is nondecreasing for $t\in [0,T_{max})$, and letting $0<\epsilon \leq  G(0)$, it follows from (\ref{p-1}) that
\begin{align*}
\|u\|_{p+1}^{\frac{2}{1-2\alpha}}=\left(\|u\|_{p+1}^{p+1}\right)^{-\sigma} \|u\|_{p+1}^{p+1}
\leq C G(t)^{-\sigma} \|u\|_{p+1}^{p+1} \leq C G(0)^{-\sigma} \|u\|_{p+1}^{p+1}
\leq C\epsilon^{-\sigma}  \|u\|_{p+1}^{p+1}.
\end{align*}
By substituting the above inequality into (\ref{doub-3}), one has
\begin{align}   \label{doub-4}
|N'(t)|^{\frac{1}{1-\alpha}}
\leq  C\left(\|u_t\|_2^2 + \epsilon^{-\sigma}  \|u\|_{p+1}^{p+1} \right), \text{\;\;\;for\;\;} t\in [0,T_{max}).
\end{align}
Since $\epsilon\leq 1$ and $\sigma>0$, then
\begin{align*}
|N'(t)|^{\frac{1}{1-\alpha}}
\leq C \epsilon^{-\sigma}  \left(\|u_t\|_2^2 + \|u\|_{p+1}^{p+1} \right),    \text{\;\;\;for\;\;} t\in [0,T_{max}),
\end{align*}
and along with (\ref{Y-2}), it follows that
\begin{align} \label{Y-3}
Y(t)^{\frac{1}{1-\alpha}}  \leq C \epsilon^{-\sigma}\left(G(t)+\|u_t\|_2^2 + \|u\|_{p+1}^{p+1}\right),  \text{\;\;\;for\;\;} t\in [0,T_{max}).
\end{align}
By virtue of (\ref{Y-0}) and (\ref{Y-3}), we obtain the desired inequality (\ref{Y-1}), which implies that $T_{max}$ is necessarily finite, i.e. the system (\ref{1.1}) blows up in finite time. In particular,
\begin{align*}
T_{max}<   \frac{1-\alpha}{\alpha}\epsilon^{-(1+\sigma)} C(k(0),p)Y(0)^{-\frac{\alpha}{1-\alpha}}
\leq  \frac{1-\alpha}{\alpha}\epsilon^{-(1+\sigma)} C(k(0),p) G(0)^{-\alpha},
\end{align*}
where the last inequality is due to (\ref{YG}).

Since $T_{max}$ is the maximum life span of the solution in the finite energy space $H^1(\O)\times L^2(\O)$ and we have shown that $T_{max}<\infty$, then it must be the case that
\begin{align}  \label{Y-12}
\limsup_{t \rightarrow T_{max}^-} \mathscr E(t) = \infty.
\end{align}
To see this, assume to the contrary that there exists $C_0>0$ such that $\mathscr E(t)\leq C_0$ for all $t\in [0,T_{max})$. Then by Theorem \ref{thm-exist} and Definition \ref{def-weak}, there exists a unique weak solution $u(t)$ on $(-\infty,T_0]$ with the regularity that $u\in C([0,T_0];H^1_0(\O))$ and $u_t\in C([0,T_0];L^2(\O))$ where $T_0>0$ depending on $C_0$ such that $T_{max}$ is not an integer multiple of $T_0$. Thus, there exists a natural number $n_0$ such that $n_0 T_0< T_{max}< (n_0+1)T_0$, and by iterating the conclusion of Theorem \ref{thm-exist} for $n_0+1$ times, the system (\ref{1.1}) admits a unique weak solution $u(t)$ on $(-\infty,(n_0+1)T_0]$, which contradicts the fact that $T_{max}$ is the maximum lifespan of the weak solution for (\ref{1.1}).

By using (\ref{to-ene}) and (\ref{EI-1}), we obtain
\begin{align*}
\frac{1}{p+1}\|u(t)\|_{p+1}^{p+1}=\mathscr E(t)
-E(t) \geq \mathscr E(t)-E(0),
\end{align*}
and along with (\ref{Y-12}), we obtain that
\begin{align}  \label{Y-13}
\limsup_{t \rightarrow T_{max}^-} \norm{u(t)}_{p+1}=\infty.
\end{align}
Finally, thanks to the Sobolev inequality $\|u(t)\|_{p+1} \leq \gamma \|\nabla u(t)\|_2$\,\,  ($p<5$ from Remark \ref{rem}), we conclude from (\ref{Y-13}) that
$$\limsup_{t\rightarrow T_{max}^-} \|\nabla u(t)\|_2 = \infty,$$
completing the proof.
\end{proof}

\bigskip

\section{Proof of Theorem \ref{thm-blow2}}  \label{S3}
This section is devoted to proving Theorem \ref{thm-blow2}, which is a finite-time blow-up result for (\ref{1.1}) under the scenario that the initial total energy $E(0)$ is nonnegative. In particular, it states that if the initial total energy $0\leq E(0)<M$ where $M>0$ is defined in (\ref{def-M}), and the initial quadratic energy $\mathscr E(0)>y_0$ where $y_0$ is defined in (\ref{con-1}), then the weak solution of (\ref{1.1}) blows up in finite time, provided the source dominates dissipation in the sense that $p> \max \{m,\sqrt{k(0)}\}$.

In order to have a better understanding of the assumptions of Theorem \ref{thm-blow2}, we shall provide the following discussions before proving the theorem. Recall that, for given $p\in (1,5]$, we set $\gamma>0$ to be the best constant for the Sobolev inequality $\|u\|_{p+1} \leq \gamma \|\nabla u\|_2$ for all $u\in H^1_0(\O)$, i.e, $\gamma^{-1}=\inf \left\{\norm{\nabla u}_2: u\in H^1_0(\O),  \|u\|_{p+1}=1 \right\}.$
Let us define the function $F: \mathbb R^+ \rightarrow \mathbb R$ by
\begin{align}     \label{Fy}
F(y)=y-\frac{1}{p+1} (2\gamma^2 y)^{\frac{p+1}{2}}.
\end{align}
We remark that the expression of $F$ originates from the right-hand side of the inequality (\ref{ene-1}) below. Since $\frac{p+1}{2}>1$, it follows that the function $F(y)$ obtains its maximum in $[0, \infty)$ at $y=y_0$, where
\begin{align}   \label{con-1}
y_0:=\frac{1}{2} \gamma^{-\frac{2(p+1)}{p-1}},
\end{align}
and the maximum value $d$ of $F(y)$ is
\begin{align}  \label{con}
d:=F(y_0)=\left(\frac{1}{2}-\frac{1}{p+1}\right) \gamma^{-\frac{2(p+1)}{p-1}}.
\end{align}

\begin{remark}   \label{rem1}
The constant $d$ defined in (\ref{con}) coincides with the mountain pass level (also the depth of the potential well \cite{Payne}), i.e., we claim
\begin{align}  \label{mount}
d=\inf_{u\in H^1_0(\O)\backslash \{0\}} \sup_{ \lambda \geq 0} J(\lambda u),
\end{align}
where we define $J(u)=\frac{1}{2}\norm{\nabla u}_2^2 - \frac{1}{p+1}\norm{u}_{p+1}^{p+1}$.
In order to verify  (\ref{mount}), we calculate
\begin{align*}
\partial_{\lambda} J(\lambda u)= \lambda \norm{\nabla u}_2^2 - \lambda^p \norm{u}_{p+1}^{p+1}, \;\;\;\;p>1.
\end{align*}
It follows that the maximum value of $J(\lambda u)$ for $\lambda\geq 0$ occurs at $\lambda_0>0$ such that $\norm{\nabla u}_2^2=\lambda_0^{p-1}\norm{u}_{p+1}^{p+1}$, for $u\not=0$. As a result,
\begin{align*}
\inf_{u\in H^1_0(\O)\backslash \{0\}} \sup_{ \lambda \geq 0} J(\lambda u)
&=\inf_{u\in H^1_0(\O)\backslash \{0\}} J(\lambda_0 u) \notag\\
&=\inf_{u\in H^1_0(\O)\backslash \{0\}}\left\{\frac{1}{2}\lambda_0^2\norm{\nabla u}_2^2
- \frac{1}{p+1}  \lambda_0^{p+1} \norm{u}_{p+1}^{p+1}\right\} \notag\\
&= \left(\frac{1}{2}-\frac{1}{p+1}\right)\inf_{u\in H^1_0(\O)\backslash \{0\}} \left(\frac{\norm{\nabla u}_2}{\|u\|_{p+1}}\right)^{\frac{2(p+1)}{p-1}} \notag\\
&= \left(\frac{1}{2}-\frac{1}{p+1}\right) \gamma^{-\frac{2(p+1)}{p-1}}=d
\end{align*}
where we have used (\ref{em}) and (\ref{con}).
\end{remark}

\smallskip

Next, we put
\begin{align} \label{ystar}
y^*:=(\sqrt{k(0)}+1)^{\frac{2}{p-1}}  (2\gamma^2)^{-\frac{p+1}{p-1}}.
\end{align}
By the assumption $k(\infty)=1$ and $k'(s)<0$ for all $s>0$, we know that $k(0)>1$, and thus, due to (\ref{ystar}) and (\ref{con-1}), one has
\begin{align}   \label{ystar2}
y^*>\frac{1}{2}\gamma^{-\frac{2(p+1)}{p-1}}=y_0.
\end{align}
Also, we define the constant $M$ by
\begin{align}   \label{defM}
M:=& \,F(y^*) \notag\\
=& \,y^*\left(\frac{p-\sqrt{k(0)}}{p+1}\right)=(\sqrt{k(0)}+1)^{\frac{2}{p-1}}  (2\gamma^2)^{-\frac{p+1}{p-1}}\left(\frac{p-\sqrt{k(0)}}{p+1}\right)>0,
\end{align}
provided $p>\sqrt{k(0)}$.
Recall that in Theorem \ref{thm-blow2} we assume that the initial total energy $E(0)<M$.

We have mentioned that the function $F(y)$ reaches its maximum at $y=y_0$, and monotone decreasing when $y>y_0$, therefore, we see that
\begin{align}    \label{defM2}
0<M=F(y^*)<F(y_0)=d,
\end{align}
due to (\ref{ystar2}) and (\ref{con}), i.e., $M$ is less than the depth of the potential well. Clearly, $M\rightarrow d^-$ as $k(0)\rightarrow 1^+$ by (\ref{defM}) and (\ref{con}), which can be interpreted as that, if the linear memory term is formally diminished in (\ref{1.1}),
then $M$ (which is the upper bound of initial energy) gets close to the mountain pass level $d$.

Now we are ready to prove Theorem \ref{thm-blow2}.

\begin{proof}
The proof draws from some ideas in \cite{G-T,Messaoudi2,Viti2}. Let us define the life span $T_{max}$ of the solution $u$ to be the supremum of all $T>0$ such that $u$ is a solution of (\ref{1.1}) on $(-\infty,T]$. We aim to show that $T_{max}$ is necessarily finite, that is, $u$ blows up in finite time.

By (\ref{to-ene}) and (\ref{em}) we have
\begin{align} \label{ene}
E(t) = \mathscr E(t) - \frac{1}{p+1} \|u(t)\|_{p+1}^{p+1} \geq \mathscr E(t) - \frac{1}{p+1} \gamma^{p+1} \|\nabla u(t)\|_2^{p+1},
\end{align}
for $t\in [0,T_{max})$. Since $\mathscr E(t)=\frac{1}{2}\left(\norm{u_t(t)}_2^2+\norm{\nabla u(t)}_2^2+\int_0^{\infty}\norm{\nabla w(t,s)}_2^2 \mu(s)ds \right)$, one has
\begin{align*}
\|\nabla u(t)\|_2 \leq \left(2\mathscr E(t)\right)^{\frac{1}{2}},      \;\;\;\;\;\text{for}\;\;t\in [0,T_{max}),
\end{align*}
and thus the inequality (\ref{ene}) implies
\begin{align}  \label{ene-1}
E(t) \geq \mathscr E(t) - \frac{1}{p+1} [2\gamma^2\mathscr E(t)]^{\frac{p+1}{2}},   \;\;\;\;\;  \text{for}\;\;     t\in [0,T_{max}).
\end{align}
Notice that, by using the function $F(y)$ defined in (\ref{Fy}), then inequality (\ref{ene-1}) takes the concise form
\begin{align}   \label{ene-2}
E(t)\geq F(\mathscr E(t)),        \;\;\;\;\;   \text{for}\;\;     t\in [0,T_{max}).
\end{align}
Recall that the continuous function $F(y)$ attains its maximum value at $y=y_0$, so it is monotone decreasing when $y>y_0$. Since we assume the initial energy $0\leq E(0)<M=F(y^*)$, there exists a unique number $y_1$ such that
\begin{align}   \label{ystar3'}
F(y_1)=E(0), \text{\;\;\;with\;\;\;} y_1>y^*>y_0>0.
\end{align}
Therefore, by using (\ref{EI-1}) and (\ref{ene-2}), we have
\begin{align}  \label{ystar3}
M>F(y_1)=E(0) \geq E(t) \geq F(\mathscr E(t)),    \;\;\;\;\;   \text{for}\;\;     t\in [0,T_{max}).
\end{align}
Since $F(y)$ is continuous and decreasing when $y>y_0$ and $\mathscr E(t)$ is also continuous, then by using the assumption that $\mathscr E(0) > y_0$, it follows from (\ref{ystar3}) that
\begin{align}    \label{ene-3}
\mathscr E(t) \geq y_1,     \;\;\;\;\;   \text{for}\;\;     t\in [0,T_{max}).
\end{align}
Consequently, by (\ref{ystar3}), (\ref{ene-3}), (\ref{defM}) and (\ref{Fy}), one has
\begin{align*}
\frac{1}{p+1} \|u(t)\|_{p+1}^{p+1} =\mathscr E(t) - E(t) \geq  y_1-F(y_1) = \frac{1}{p+1}  (2\gamma^2y_1)^{\frac{p+1}{2}},
\end{align*}
which can be reduced to
\begin{align}   \label{upp}
\|u(t)\|_{p+1}^{p+1} \geq  (2\gamma^2 y_1)^{\frac{p+1}{2}},  \text{\;\;\;\;\;for\;\;}   t\in [0,T_{max}).
\end{align}

Now we set $\mathcal G(t) = M - E(t)>0$ and $N(t)=\frac{1}{2}\|u(t)\|_2^2$ for $t\in [0,T_{max})$. We aim to show that
\begin{align} \label{def-Y'}
\mathcal Y(t)=\mathcal G(t)^{1-\alpha} + \epsilon N'(t)
\end{align}
blows up in finite time, for some $\alpha\in (0,1)$ and $\epsilon>0$, which will be selected later. By differentiating both sides of (\ref{def-Y'}) and using (\ref{doub}), one has
\begin{align}    \label{tilY}
\mathcal Y'(t) = (1-\alpha) \mathcal G(t)^{-\alpha} \mathcal G'(t) + \epsilon \Big(&\|u_t\|_2^2  - \|\nabla u\|_2^2 + \int_0^{\infty} k'(s) \int_{\O} \nabla u(t) \cdot \nabla w(t,s) dx ds   \notag\\
&- \int_{\O}  |u_t|^{m-1} u_t u dx + \|u\|_{p+1}^{p+1}\Big).
\end{align}
By using (\ref{ene-3}) and (\ref{defM}) we obtain
\begin{align*}
\mathcal G(t) = M - E(t) &= M - \mathscr E(t) + \frac{1}{p+1} \|u(t)\|_{p+1}^{p+1}  \notag\\
&\leq y^*\left(\frac{p-\sqrt{k(0)}}{p+1}\right) - y_1 + \frac{1}{p+1} \|u(t)\|_{p+1}^{p+1}  \notag\\
&=(y^*-y_1)-y^*\left(\frac{\sqrt{k(0)}+1}{p+1}\right) + \frac{1}{p+1} \|u(t)\|_{p+1}^{p+1} \notag\\
&<-y^*\left(\frac{\sqrt{k(0)}+1}{p+1}\right) + \frac{1}{p+1} \|u(t)\|_{p+1}^{p+1},
\end{align*}
since $y_1>y^*$. The last inequality can be expressed as
\begin{align}    \label{upp-1}
\|u(t)\|_{p+1}^{p+1}>(p+1)\mathcal G(t) + y^*(\sqrt{k(0)}+1).
\end{align}
Also, since $\mathcal G(t)=M-E(t)$ then by the energy identity (\ref{EI-1}), we have
\begin{align}  \label{tilG}
\mathcal G'(t) = - E'(t)= \|u_t\|_{m+1}^{m+1}-\frac{1}{2}\int_0^{\infty} \|\nabla w(t,s)\|_2^2  \mu'(s)  ds  \geq \|u_t\|_{m+1}^{m+1},
\end{align}
where we have used the assumption $\mu'(s)\leq 0$. Note that (\ref{tilG}) shows that $\mathcal G(t)$ is nondecreasing for $t\in [0,T_{max})$.

By employing (\ref{upp-1}) as well as (\ref{tilG}), we can carry out the same estimate used in (\ref{ut}), (\ref{ut-0})-(\ref{ut-1}) to obtain
\begin{align}  \label{ut-2}
\left|\int_{\O} u |u_t|^{m-1} u_t dx \right| \leq \lambda \mathcal G(0)^{\frac{1}{p+1}-\frac{1}{m+1}} \|u\|_{p+1}^{p+1} +  C_{\lambda}  \mathcal G(0)^{\frac{1}{p+1}-\frac{1}{m+1}+\alpha} \mathcal G(t)^{-\alpha}  \mathcal G'(t),
\end{align}
where as before, we choose $0<\alpha<\frac{1}{m+1}-\frac{1}{p+1}$, and $\lambda$ is a positive constant which will be selected later.

By applying the estimates (\ref{ut-2}) and (\ref{ks}) to the identity (\ref{tilY}), we have
\begin{align}  \label{Y-4}
\mathcal Y'(t) \geq &  \left[1-\alpha- \epsilon  C_{\lambda}  \mathcal G(0)^{\frac{1}{p+1}-\frac{1}{m+1}+\alpha}\right] \mathcal G(t)^{-\alpha} \mathcal G'(t) \notag\\
&+ \epsilon\Big[ \|u_t\|_2^2 - \left(\frac{k(0)-1}{4\delta}+1\right) \|\nabla u\|_2^2 - \delta \int_0^{\infty}   \|\nabla w(t,s)\|^2_2   \mu(s) ds \notag\\
&\;\;\;\;\;\;\;\;\;+(1-\lambda \mathcal G(0)^{\frac{1}{p+1}-\frac{1}{m+1}} )\|u\|_{p+1}^{p+1}\Big],  \text{\;\;\;for\;\;} t\in [0,T_{max}).
\end{align}
By (\ref{to-ene}) we see that
\begin{align*}
\int_0^{\infty} \norm{\nabla w(t,s)}_2^2 \mu(s) ds  =  2E(t) - \|u_t\|_2^2 - \|\nabla u\|_2^2  + \frac{2}{p+1}\|u\|_{p+1}^{p+1},
\end{align*}
which can be substituted into (\ref{Y-4}) to obtain,
\begin{align}   \label{Y-5}
\mathcal Y'(t) \geq &\left[1-\alpha- \epsilon  C_{\lambda}  \mathcal G(0)^{\frac{1}{p+1}-\frac{1}{m+1}+\alpha}\right] \mathcal G(t)^{-\alpha} \mathcal G'(t) + \epsilon\Bigg[ \Big(\delta-\frac{k(0)-1}{4\delta}-1\Big) \|\nabla u\|_2^2    \notag\\
& + (1+\delta)\|u_t\|_2^2  +\Big(1-\frac{2\delta}{p+1}-\lambda \mathcal G(0)^{\frac{1}{p+1}-\frac{1}{m+1}} \Big)\|u\|_{p+1}^{p+1}  - 2\delta E(t) \Bigg],
\end{align}
for $t\in [0,T_{max})$. Now we choose
\begin{align*}
\delta=\frac{\sqrt{k(0)}+1}{2},
\end{align*}
so that $\delta-\frac{k(0)-1}{4\delta}-1=0$, and thus the inequality (\ref{Y-5}) takes the form
\begin{align}  \label{Y-6}
\mathcal Y'(t) \geq &\left[1-\alpha- \epsilon  C_{\lambda}  \mathcal G(0)^{\frac{1}{p+1}-\frac{1}{m+1}+\alpha}\right] \mathcal G(t)^{-\alpha} \mathcal G'(t) + \epsilon\Bigg[  \frac{\sqrt{k(0)}+3}{2}\|u_t\|_2^2 \notag\\
&   +\left(\frac{p-\sqrt{k(0)}}{p+1}-\lambda \mathcal G(0)^{\frac{1}{p+1}-\frac{1}{m+1}} \right)\|u\|_{p+1}^{p+1}  - \left(\sqrt{k(0)}+1\right) E(t) \Bigg],
\end{align}
for $t\in [0,T_{max})$, where we require $p>\sqrt{k(0)}$.

Next, we aim to show that
$$\frac{p-\sqrt{k(0)}}{p+1}  \|u\|_{p+1}^{p+1}  - \left(\sqrt{k(0)}+1\right) E(t) > c\|u\|_{p+1}^{p+1}, \text{\;\;\;for all\;\;}t\in [0,T_{max}),$$
for some $c>0$.
For the sake of convenience, we put
\begin{align}  \label{C0}
C_0=(2\gamma^2 y_1)^{\frac{p+1}{2}}.
\end{align}
Then,  it follows from(\ref{EI-1}), (\ref{ystar3'}) and (\ref{Fy}) that,
\begin{align}   \label{Fy-2}
E(t)\leq E(0)=F(y_1)=y_1-\frac{1}{p+1} (2\gamma^2 y_1)^{\frac{p+1}{2}}=y_1-\frac{1}{p+1}C_0,
\end{align}
for $t\in [0,T_{max})$.

Now we split the term $\frac{p-\sqrt{k(0)}}{p+1}\|u\|_{p+1}^{p+1}$ into two positive parts:
\begin{align}  \label{split}
\frac{p-\sqrt{k(0)}}{p+1}\|u\|_{p+1}^{p+1}    =&\left(\frac{p-\sqrt{k(0)}}{p+1}-\frac{C_0-(\sqrt{k(0)}+1)y_1}{2C_0}\right)\|u\|_{p+1}^{p+1}  \notag\\
&+\frac{C_0-(\sqrt{k(0)}+1)y_1}{2C_0} \|u\|_{p+1}^{p+1}.
\end{align}
The fact that the two terms on the right-hand side of (\ref{split}) are both positive comes from the following straightforward calculations. Indeed,
by (\ref{C0}) and the fact that $y_1>y_0=\frac{1}{2} \gamma^{-\frac{2(p+1)}{p-1}}$ as well as the assumption $p>\sqrt{k(0)}>1$, we compute
\begin{align} \label{posi}
&\frac{p-\sqrt{k(0)}}{p+1}-\frac{C_0-(\sqrt{k(0)}+1)y_1}{2C_0} \notag\\
&=\frac{C_0(p-2\sqrt{k(0)}-1)+(\sqrt{k(0)}+1)y_1(p+1)}{2C_0(p+1)} \notag\\
&>\frac{y_1\left[(2\gamma^2)^{\frac{p+1}{2}}y_0^{\frac{p-1}{2}}(p-2\sqrt{k(0)}-1)+(\sqrt{k(0)}+1)(p+1)\right]}{2C_0(p+1)}  \notag\\
&=\frac{y_1\left[3(p-\sqrt{k(0)})+p\sqrt{k(0)}-1\right]}{2C_0(p+1)}>0.
\end{align}
Also, thanks  to (\ref{C0}) and the fact that $y_1>y^*=(\sqrt{k(0)}+1)^{\frac{2}{p-1}}  (2\gamma^2)^{-\frac{p+1}{p-1}}$, we see that
\begin{align}  \label{bdefc}
C_0-(\sqrt{k(0)}+1)y_1>\left[ (2\gamma^2)^{\frac{p+1}{2}} (y^*)^{\frac{p-1}{2}}- (\sqrt{k(0)}+1) \right]y_1=0.
\end{align}
Thus, we can define the positive constant $c$ as
\begin{align}  \label{def-c}
c:=\frac{C_0-(\sqrt{k(0)}+1)y_1}{2C_0}>0.
\end{align}
Applying (\ref{posi}) and (\ref{def-c}) along with the fact that $\|u(t)\|_{p+1}^{p+1} \geq  C_0$ for $t\in [0,T_{max})$ from (\ref{upp}), we obtain from (\ref{split}) that
\begin{align}   \label{psk}
\frac{p-\sqrt{k(0)}}{p+1}\|u\|_{p+1}^{p+1}\geq\left(\frac{p-\sqrt{k(0)}}{p+1}-\frac{C_0-(\sqrt{k(0)}+1)y_1}{2C_0}\right)C_0+c \|u\|_{p+1}^{p+1}.
\end{align}
By using (\ref{Fy-2}) and (\ref{psk}), we calculate
\begin{align}   \label{pk1}
&\frac{p-\sqrt{k(0)}}{p+1}\|u\|_{p+1}^{p+1}  - (\sqrt{k(0)}+1) E(t) \notag\\
&\geq  \left(\frac{p-\sqrt{k(0)}}{p+1}-\frac{C_0-(\sqrt{k(0)}+1)y_1}{2C_0}\right)C_0 +c\|u\|_{p+1}^{p+1} \notag\\
& \hspace{1 in}- (\sqrt{k(0)}+1) \left(y_1 - \frac{1}{p+1} C_0\right)  \notag\\
&=\frac{C_0-(\sqrt{k(0)}+1)y_1}{2}+ c\|u\|_{p+1}^{p+1}  \notag\\
&> c\|u\|_{p+1}^{p+1},
\end{align}
where the last inequality follows from (\ref{bdefc}).

Applying (\ref{pk1}) to (\ref{Y-6}) yields
\begin{align}  \label{Y-7}
\mathcal Y'(t) > &\left[1-\alpha- \epsilon  C_{\lambda}  \mathcal G(0)^{\frac{1}{p+1}-\frac{1}{m+1}+\alpha}\right] \mathcal G(t)^{-\alpha} \mathcal G'(t)\notag\\
 &+\epsilon\left[  \frac{\sqrt{k(0)}+3}{2}\|u_t\|_2^2 +\left(c-\lambda \mathcal G(0)^{\frac{1}{p+1}-\frac{1}{m+1}} \right)\|u\|_{p+1}^{p+1}  \right].
\end{align}
Now, we choose $\lambda>0$ such that $\lambda \mathcal G(0)^{\frac{1}{p+1}-\frac{1}{m+1}}=\frac{c}{2}$ and select $\epsilon>0$ sufficiently small so that $\epsilon  C_{\lambda}  \mathcal G(0)^{\frac{1}{p+1}-\frac{1}{m+1}+\alpha}\leq 1-\alpha$, we obtain from (\ref{Y-7}) that
\begin{align}  \label{Y-8}
\mathcal Y'(t) >  \frac{\epsilon}{2}\left[  \left(\sqrt{k(0)}+3\right)\|u_t(t)\|_2^2 + c \|u(t)\|_{p+1}^{p+1}  \right], \text{\;\;\;\;\;for\;\;} t\in [0,T_{max}).
\end{align}

Combining the estimates (\ref{Y-8}) and (\ref{upp-1}) yields that
\begin{align}   \label{Y-9}
\mathcal Y'(t) >  \frac{\epsilon}{2}\left[  \left(\sqrt{k(0)}+3\right)\|u_t(t)\|_2^2 + \frac{c}{2} \|u(t)\|_{p+1}^{p+1} + \frac{c}{2}(p+1) \mathcal G(t)  \right]>0,
\end{align}
for $t\in [0,T_{max})$, where the last inequality is due to the fact $\mathcal G(t)=M-E(t)>0$.

Notice that $\mathcal Y(0)=\mathcal G(0)^{1-\alpha}+\epsilon N'(0)$, and if $N'(0)<0$, then we shall further impose the restriction  $0<\epsilon\leq -\frac{\mathcal G(0)^{1-\alpha}}{2N'(0)}$ so that $\mathcal Y(0)\geq \frac{1}{2} \mathcal G(0)^{1-\alpha}$. Since $\mathcal Y(t)$ is increasing on $[0,T_{max})$ by virtue of (\ref{Y-9}), it follows that
\begin{align}    \label{Y-11}
\mathcal Y(t)\geq \mathcal Y(0)\geq \frac{1}{2} \mathcal G(0)^{1-\alpha}>0, \text{\;\;\;\;\;for\;\;} t\in [0,T_{max}).
\end{align}

Also, by following the estimates (\ref{Y-2})-(\ref{Y-3}) in the proof of Theorem \ref{thm-blow},
and by imposing the additional restrictions on $\alpha$ and $\epsilon$, namely, $0<\alpha<\frac{p-1}{2(p+1)}$
and $0<\epsilon \leq \min\{\mathcal G(0),1\}$,  we obtain
\begin{align}   \label{Y-10}
\mathcal Y(t)^{\frac{1}{1-\alpha}} \leq C \epsilon^{-\sigma} \left(\mathcal G(t)  + \norm{u_t(t)}_2^2 +\norm{u(t)}_{p+1}^{p+1} \right),
\text{\;\;\;\;\;for\;\;} t\in [0,T_{max}),
\end{align}
where $\sigma=1-\frac{2}{(1-2\alpha)(p+1)}>0$.

By taking account of inequalities (\ref{Y-9}) and (\ref{Y-10}), we see that
\begin{align*}
\mathcal Y'(t) \geq \epsilon^{1+\sigma} C(k(0),p,E(0)) \mathcal Y(t)^{\frac{1}{1-\alpha}},    \text{\;\;\;\;for\;\;} t\in [0,T_{max}),
\end{align*}
and since $\frac{1}{1-\alpha}>1$, we conclude that $T_{max}$ is necessarily finite. More precisely,
\begin{align*}
T_{max}< \frac{1-\alpha}{\alpha}\epsilon^{-(1+\sigma)}C(k(0),p, E(0)) \mathcal Y(0)^{-\frac{\alpha}{1-\alpha}}
\leq \frac{1-\alpha}{\alpha}\epsilon^{-(1+\sigma)}C(k(0),p, E(0)) \mathcal G(0)^{-\alpha},
\end{align*}
where the last inequality comes from (\ref{Y-11}). Finally, by adopting the same argument as in the proof of Theorem \ref{thm-blow}, we conclude that $\limsup_{t\rightarrow T_{max}^-}\norm{\nabla u(t)}_2=\infty$ and $\limsup_{t\rightarrow T_{max}^-}\norm{u(t)}_{p+1}=\infty$.
\end{proof}

\vspace{0.05 in}

We remark that, if the condition $\mathscr E(0) > y_0$ in Theorem \ref{thm-blow2} is replaced with a different assumption $\norm{u_0(0)}_{p+1}^{p+1} > \norm{\nabla u_0(0)}_2^2$, then the solution still blows up in finite time. Specifically, we have the following corollary of Theorem \ref{thm-blow2}.
Please refer to \cite{ACCRT,BRT,GR2} for comparable results concerning blow-up of wave equations with nonlinear sources and damping (but without memory) by using a different approach which involves a contradiction argument.
\begin{corollary}  \label{cor1}
In addition to the validity of the Assumption \ref{ass}, we assume that $p>\max\{m,\sqrt{k(0)}\}$. Also, we suppose that $\norm{u_0(0)}_{p+1}^{p+1} > \norm{\nabla u_0(0)}_2^2$, and
\begin{align*}
0\leq E(0)<M:=(\sqrt{k(0)}+1)^{\frac{2}{p-1}}  (2\gamma^2)^{-\frac{p+1}{p-1}}\left(\frac{p-\sqrt{k(0)}}{p+1}\right).
\end{align*}
Then the weak solution $u$ of the system (\ref{1.1}) blows up in finite time. More precisely, $\limsup_{t\rightarrow T_{\max}^-} \|\nabla u(t)\|_2=\infty$, for some $T_{\max}\in (0,\infty)$.
\end{corollary}

\begin{proof}
It is sufficient to show that the condition $\norm{u_0(0)}_{p+1}^{p+1} > \norm{\nabla u_0(0)}_2^2$ implies that $\mathscr E(0) >  y_0$. To this end, let us recall that $J(u)=\frac{1}{2} \norm{\nabla u}_2^2 - \frac{1}{p+1} \norm{u}_{p+1}^{p+1}$, then the maximum value of $J(\lambda u_0(0))$ for $\lambda \geq 0$ occurs at $\lambda_0$ such that $\norm{\nabla u_0(0)}_2^2 =  \lambda_0^{p-1} \norm{u_0(0)}_{p+1}^{p+1}$.
Also, since $\norm{u_0(0)}_{p+1}^{p+1} > \norm{\nabla u_0(0)}_2^2$, it follows that $\lambda_0<1$.
Consequently, by (\ref{mount}), we obtain
\begin{align*}
d\leq \sup_{\lambda\geq 0} J(\lambda u_0(0))= J(\lambda_0 u_0(0))
&=\frac{1}{2} \lambda_0^2 \norm{\nabla u_0(0)}_2^2  -  \frac{1}{p+1} \lambda_0^{p+1}  \norm{u_0(0)}_{p+1}^{p+1}   \notag\\
&=\lambda_0^2  \left(\frac{1}{2}-\frac{1}{p+1}\right) \norm{\nabla u_0(0)}_2^2  \notag\\
&< \left(\frac{1}{2}-\frac{1}{p+1}\right) \norm{\nabla u_0(0)}_2^2.
\end{align*}
This shows that
\begin{align*}
\norm{\nabla u_0(0)}_2^2 > \frac{2(p+1)}{p-1}d=\gamma^{-\frac{2(p+1)}{p-1}}=2y_0,
\end{align*}
where we have used (\ref{con-1}) and (\ref{con}). Thus, $\mathscr E(0)>y_0$.

\end{proof}

\begin{remark}
Since it has been shown that the condition $\norm{u_0(0)}_{p+1}^{p+1} > \norm{\nabla u_0(0)}_2^2$ implies that $\mathscr E(0) >  y_0$, one realizes that the assumptions of Theorem \ref{thm-blow2} are weaker than the assumptions of Corollary \ref{cor1}, which often appear in the literature (see for instance \cite{ACCRT,BRT,GR1}).
Also, the assumption that $\mathscr E(0) >  y_0$ contains all of the past history from $-\infty$ to $0$, which is a more appropriate assumption for a system with delay, compared to the condition that $\norm{u_0(0)}_{p+1}^{p+1} > \norm{\nabla u_0(0)}_2^2$ which involves partial information of the initial datum only at $t=0$. On the other hand, it is worth mentioning that, if the initial datum satisfies that $0\leq E(0)<M$ and $\norm{u_0(0)}_{p+1}^{p+1} > \norm{\nabla u_0(0)}_2^2$, then $\norm{u(t)}_{p+1}^{p+1} > \norm{\nabla u(t)}_2^2$ for all time $t$ before the formation of singularity. Indeed, if there exists $t_1>0$ such that $\norm{u(t_1)}_{p+1}^{p+1} = \norm{\nabla u(t_1)}_2^2$, then
\begin{align*}
J(u(t_1))=\frac{1}{2} \norm{\nabla u(t_1)}_2^2 - \frac{1}{p+1} \norm{u(t_1)}_{p+1}^{p+1}=\left(\frac{1}{2}-\frac{1}{p+1}\right)\norm{\nabla u(t_1)}_2^2.
\end{align*}
It follows that
\begin{align*}
\norm{\nabla u(t_1)}_2^2 \leq \frac{2(p+1)}{p-1}J(u(t_1))\leq \frac{2(p+1)}{p-1}E(t_1)\leq \frac{2(p+1)}{p-1}E(0).
\end{align*}
Hence, by Sobolev inequality $\norm{u}_{p+1} \leq \gamma \norm{\nabla u}_2$, we obtain
\begin{align}   \label{contr}
\norm{u(t_1)}_{p+1}^{p+1} &\leq \gamma^{p+1} \left(\norm{\nabla u(t_1)}_2^2\right)^{\frac{p-1}{2}}  \norm{\nabla u(t_1)}_2^2 \notag\\
&\leq \gamma^{p+1} \left(\frac{2(p+1)}{p-1}E(0)\right)^{\frac{p-1}{2}} \norm{\nabla u(t_1)}_2^2 \notag\\
&<\norm{\nabla u(t_1)}_2^2
\end{align}
where the last inequality follows from $E(0)<M<d=\left(\frac{1}{2}-\frac{1}{p+1}\right)\gamma^{-\frac{2(p+1)}{p-1}}$. However,  (\ref{contr}) contradicts the assumption that $\norm{u(t_1)}_{p+1}^{p+1} = \norm{\nabla u(t_1)}_2^2$, and so, it must be the case that
$\norm{u(t)}_{p+1}^{p+1} > \norm{\nabla u(t)}_2^2$, for all $t\in [0,T_{max}) $.

\end{remark}

\bibliographystyle{amsplain}

\end{document}